\documentclass{amsart}
\usepackage{amsmath,amssymb,amsfonts,amsthm}
\usepackage{amssymb}
\usepackage{pgf}
\usepackage{todonotes}
\usepackage{epsfig}
\usepackage{mathrsfs}
\usepackage{verbatim}
\newtheorem{theorem}{Theorem}[section]
\newtheorem{lemma}[theorem]{Lemma}
\newtheorem{proposition}[theorem]{Proposition}

\def\Xint#1{\mathchoice
{\XXint\displaystyle\textstyle{#1}}%
{\XXint\textstyle\scriptstyle{#1}}%
{\XXint\scriptstyle\scriptscriptstyle{#1}}%
{\XXint\scriptscriptstyle\scriptscriptstyle{#1}}%
\!\int}
\def\XXint#1#2#3{{\setbox0=\hbox{$#1{#2#3}{\int}$ }
\vcenter{\hbox{$#2#3$ }}\kern-.6\wd0}}

\def\dashint{\Xint-}
\theoremstyle{definition}

\newtheorem{example}[theorem]{Example}

\theoremstyle{remark}
\newtheorem{remark}[theorem]{Remark}

\numberwithin{equation}{section}

\newcommand{\conv}{\mathrm{conv}}

\newcommand{\z}{\zeta}
\newcommand{\R}{\mathbb{R}}
\newcommand{\N}{\mathrm{N}}
\newcommand{\C}{\mathbb{C}}
\newcommand{\D}{\mathbb{D}}

\newcommand{\X}{\mathrm{X}}
\newcommand{\A}{\mathrm{A}}
\newcommand{\bb}{\mathrm{E}}
\newcommand{\cc}{\mathrm{C}}

\newcommand{\Y}{\mathrm{Y}}
\newcommand{\Z}{\mathbb{Z}}
\newcommand{\B}{\mathbf{B}}

\renewcommand{\S}{\mathbf{S}}
\newcommand{\T}{\mathbb{T}}

\newcommand{\1}{\boldsymbol{1}}

\title[Convex-transitive Douglas algebras]{Convex-transitive Douglas algebras}

\author{Mar\'ia J. Mart\'in and Jarno Talponen}
\address{
   M. J. Mart\'{\i}n and J. Talponen\\
   Department of Physics and Mathematics, University of Eastern Finland, P.O.~Box~111, FI-80101 Joensuu
      \\
   Finland}
\email{maria.martin@uef.fi\\ talponen@iki.fi}

\date{\today}

\begin{document}
\maketitle

\begin{abstract}
The convex-transitivity property can be seen as a convex generalization of the almost transitive (or quasi-isotropic) group action of the isometry group of a Banach space on its unit sphere. We will show that certain Banach algebras, including conformal invariant Douglas algebras, are weak-star convex-transitive.
Geometrically speaking, this means that the investigated spaces are highly symmetric. Moreover, it turns out that the symmetry property is satisfied by using
only `inner' isometries, i.e. a subgroup consisting of isometries which are homomorphisms on the algebra. In fact, weighted composition operators arising from function theory on the unit disk will do. Some interesting 
examples are provided at the end.
\end{abstract}

\section{Introduction}

In this paper we study examples of a Nevanlinna class type function spaces having a rich isometry group, which, in a sense, comes close to  acting transitively on the unit sphere. To facilitate the discussion, let us recall some basic notions. We denote the closed unit ball of a Banach space $\X$ by $\B_{\X}$ and the unit sphere of $\X$ by $\S_{\X}$. A Banach space $\X$ is called \emph{transitive} if for each $x\in \S_{\X}$ the \emph{orbit}
$\mathcal{G}_{\X}(x)\stackrel{\cdot}{=}\{T(x)|\ T\colon \X\rightarrow \X\ \mathrm{is\ an\ isometric\ isomorphism}\}$ is $\S_{\X}$. In other words, the isometry group acts transitively on the unit sphere. If $\overline{\mathcal{G}_{\X}(x)}=\S_{\X}$ (resp. $\overline{\conv}(\mathcal{G}_{\X}(x))=\B_{\X}$)
for all $x\in\S_{\X}$, then $\X$ is called \emph{almost transitive} (resp. \emph{convex-transitive}).
It was first reported by Pelczynski and Rolewicz in 1962 \cite{PR} that the space $L^{p}$ is almost transitive for
$p\in [1,\infty)$ and convex-transitive for $p=\infty$ (see also \cite{Rol}). These concepts are motivated by
\emph{Mazur's rotation problem} appearing in \cite[p.~242]{Ba}, which remains open.
We refer to \cite{BR2} for a survey and discussion on the matter.

By applying categorical methods one can verify the existence of a rich class of almost transitive Banach spaces (see e.g. the above survey). Here we provide
examples of concrete $\omega^*$-convex-transitive complex Banach algebras modeled on the unit disk. Specimens of convex-transitive spaces can be found for example in \cite{Ca5}, \cite{Ca6}, \cite{GJK}, \cite{Rambla}, \cite{RT}, \cite{Talponen1}, and \cite{Talponen2}.
\par\smallskip

The \emph{pointwise multiplication operator $M_\psi$ with symbol $\psi$} in a function space $\X$ is defined by the formula
$M_\psi f= \psi f$ for all $f\in\X$.  For a characterization of isometric (not necessarily surjective) pointwise multiplication operators on several spaces of analytic functions in the unit disk $\D$, see \cite{ADMV}. Let $\varphi$ be an analytic function in the unit disk with $\varphi(\D)\subset \D$. The \emph{composition operator $C_\varphi$ with symbol} $\varphi$ is given by $C_\varphi f=f\circ\varphi$. Regarding the problem of determining the isometric composition operators on spaces of analytic functions, we refer  the reader to the works \cite{CW}, \cite{La}, \cite{M1}, \cite{M2}, or \cite{M3}, for instance. Whenever we have an operator of the form $T=M_\psi\circ C_\varphi$, we say that $T$ is a \emph{weighted composition operator}. Forelli \cite{Fo} showed that isometries (surjective or not) for the Hardy spaces $H^p (\D)$, $1\leq p<\infty$, $p\neq 2$, are precisely operators of this form (see \cite[Theorem~4.2.8]{FJ1}).
\par\smallskip
The space $H^\infty (\D)$ is of course well-known and its convex-transitivity appears like a natural
question. On the other hand, its surjective isometries have a very restricted form, namely, they are weighted composition operators
\begin{equation}\label{eq: char}
Tf(z) =\alpha f(\phi(z))\,,\quad z\in\D ,
\end{equation}
where $\alpha$ is a unimodular complex number and $\phi$ is a conformal map from $\D$ onto itself (see \cite[Thm. 4.2.2]{FJ1}). Thus it is easy to see that the space $H^\infty (\D)$ is \emph{not} convex-transitive.
What is required to increase the amount of symmetries of $H^\infty (\D)$ to make the resulting space convex-transitive is allowing `abstract divisions' by functions having possibly zeros on the disk. \emph{Douglas algebras} can be regarded as examples of such spaces. It is a beautiful result that every closed algebra between $H^\infty(\D)$ and $L^\infty(\partial\D)$ is actually a Douglas algebra. We refer  the reader to \cite[Ch. IX]{G} and the references therein for a comprehensive exposition on the topic.

The isometries studied in this paper are `\,inner\,' in the sense that they arise naturally from the function theory in the unit disk and they are multiplicative with respect to the algebraic structure of the spaces. The fact that the isometries are multiplicative implies that their adjoints retain the maximal ideal space of $H^\infty (\D)$ invariant as a subset of the dual space.

Some relevant results (for instance, Theorem~\ref{thm: dense} below) arising from the function theory appears to have been overlooked in the Banach space rotation branch, probably because of the distance of the fields. One of the aims in this paper is pointing out some of these existing connections.

\section{Preliminaries}
As was mentioned before, the open unit disk of the complex plane $\C$ will be denoted by $\D$. Its boundary, the unit circle, is $\partial \D$.
When we wish to emphasize the multiplicative structure of this set, being (additively) the group $\R / \Z$, we denote $\partial\D$ by $\T$.
\par
All Banach spaces here are regarded over the complex field. Recall that
\[
\mathcal{G}_{\X}\stackrel{\cdot}{=}\{T|\ T\colon \X\rightarrow \X\ \mathrm{is\ an\ isometric\ isomorphism}\}
\]
denotes the isometry group of $\X$. These isometries are often called \emph{rotations}.  We refer to \cite{Ca6} and \cite{FJ1} for suitable background information.
We call an element $x\in \S_{\X}$ a \emph{point of convex-transitivity} if $\overline{\conv}(\mathcal{G}_\X(x))=\B_{\X}$ (in the literature a term `big point' is also used).
If $\overline{\conv}^{\omega^*}(\mathcal{G}_\X(x))=\B_{\X}$ for each $x\in \S_X$, then, following \cite{BR2}, we call the space
$\omega^*$\emph{-convex-transitive} (cf. \cite{BR1,BRW}).
\par\smallskip
The following elementary fact is applied here frequently.
\begin{lemma}\label{lm: tri}
Suppose that $y\in \overline{\conv}(\mathcal{G}_\X(x))$ and $z\in \overline{\conv}(\mathcal{G}_{\X}(y))$. Then $z\in \overline{\conv}(\mathcal{G}_\X(x))$. The same conclusion holds analogously for the weak-star closures.
\end{lemma}
\begin{proof}
The first part of the statement follows by applying multiple times the triangle inequality, or, alternatively, by observing that the set $\overline{\conv}(\mathcal{G}_\X(x))$ is invariant under all rotations $T \in \mathcal{G}_\X$.
The latter part of the statement follows in the same vein.
\end{proof}

We denote by $m$ the Lebesgue measure on the unit circle $\partial \D$.
In what follows, the weak-star topology \emph{always} refers to the relative topology inherited from the weak-star topology of
$L^\infty = L ^\infty (\partial \D, m)$. As customary in connection with the Douglas algebras, we will abbreviate $\cc=C(\partial \D)$, the complex space of continuous
functions on the unit circle.

\subsection{Decomposition of bounded analytic functions}
The Hardy space $H^\infty=H^\infty(\D)$ consists of those analytic functions $f$ in the unit disk with $$\|f\|_\infty=\sup_{|z|<1} |f(z)|<\infty\,.$$ It is well-known that given any $f\in H^\infty$, the radial limit
\begin{equation}\label{eq-radiallimit}
\widetilde f(e^{i\theta})=\lim_{r\to 1^-} f(re^{i\theta})
\end{equation}
exists for a.e. $\theta\in [0, 2\pi)$, with $\widetilde f\in L^\infty=L^\infty(\partial\D)$ and $\|\widetilde f\|_{L^\infty(\partial\D)}=\|f\|_\infty$. Another important result related to $H^\infty$ functions is the following (see \cite[Ch.~2]{D}): any $f\in H^\infty$ can be written as a product $f=BSF$, where $B$ is a \emph{Blaschke product}, this is, a function of the form
\begin{equation*}\label{eq-Blaschkeprod}
B(z)=\lambda \prod_{n=1}^\infty \frac{|a_n|}{a_n} \frac{a_n-z}{1-\overline{a_n}z}\,,
\end{equation*}
with $\lambda\in\partial\D$ and where $\{a_n\}$ is a sequence of points in $\D$ ($|a_n|/a_n\equiv 1$ is understood whenever $a_n=0$) satisfying $\sum_{n=1}^\infty (1-|a_n|)<\infty$. The set $\{a_n\}$ might be finite and in this case we call the function $B$ a \emph{finite Blaschke product}.
\par
The factor $S$ in the decomposition is a \emph{singular inner function}, i.e.
\[
S(z)=\exp\left\{-\int_0^{2\pi}  \frac{e^{i\theta}+z}{e^{i\theta}-z}\, d\mu(t) \right\}\,,
\]
where $\mu$ is a bounded nondecreasing function on $\partial\D$ with $\mu'(t)=0$ a.e. Finally, $F$ is an \emph{outer function}, this is
\begin{equation}\label{eq-outer}
F(z)=\exp\left\{\int_0^{2\pi}  \frac{e^{it}+z}{e^{i\theta}-z}\, \log |\widetilde{f}(e^{it})|\, \frac{dt}{2\pi} \right\}\,.
\end{equation}
\par\medskip
Any function of the form $BS$, where $B$ is a Blaschke product and $S$ is a singular inner function is called \emph{inner function}, that is,
a function $f\in H^\infty$ with $|\widetilde{f}|=1$ a.e.
\par
It was proved by Frostman \cite{Frostman} that every inner function can be uniformly approximated by Blaschke products. Carath\'{e}od\-ory \cite{Caratheodory} showed that any function $f\in H^\infty$ with norm $\|f\|_\infty<1$ can be approximated locally uniformly by finite Blaschke products. The following result can be understood as a generalization of Carath\'{e}od\-ory's result (see \cite{Mar}).
\begin{theorem}[(D. Marshall)]\label{thm: dense}
The norm-closed convex hull of Blaschke products is $\B_{H^\infty (\D)}$.
\end{theorem}

\subsection{The Nevanlinna class}
Let
\[
\log^{+}(x)=\left\{
\begin{array}{cc}
\log x\,, & x\geq 1\\
0\,, & 0\leq x <1
\end{array}
\right.\,.
\]
\par
The Nevanlinna class $N$, also known as the class of functions of bounded characteristic, is the family of analytic functions $f$ in the unit disk for which the integrals
\[
\int_0^{2\pi}\log^{+}|f(r e^{it})|  \frac{dt}{2\pi}
\]
are bounded for $r<1$. Since $\log^{+}|f|$ is subharmonic whenever $f$ is analytic, these integrals always increase with $r$.
\par
It was proved by F. and R. Nevanlinna (see \cite[Thm. 2.1]{D}) that an analytic function $f$ belongs to the Nevanlinna class if and only if it is the quotient of two bounded analytic functions. The importance of this result (as mentioned on \cite[p. 17]{D}) is that it allows properties of funtions in $N$ to be deduced from the corresponding properties of bounded analytic functions. For instance, every Nevanlinna function has radial limit $\widetilde{f}(e^{it})$ a.e., and $\log |\widetilde{f}(e^{it})|$ is integrable unless $f\equiv 0$.
\par
The factorization of functions in the class $N$ is the following \cite[Thm. 2.9]{D}: every function $f\not\equiv 0$ in the class $N$ can be expressed in the form $f=B[S_1/S_2]F$, where $B$ is a Blaschke product, $S_1$ and $S_2$ are singular inner functions, and $F$ is an outer function for the Nevanlinna class; this is, $F$ has the form \eqref{eq-outer} with $\log |\widetilde{f}(e^{it})|\in L^1(\partial\D)$.
\par
It is not difficult to prove that the Nevanlinna class $N$ contains the Hardy spaces $H^p$ for every $p>0$. Recall that $H^p=H^p(\D)$ is defined as the set of analytic functions $f$ in the unit disk satisfying
\[
 \sup_{0\leq r<1} \int_0^{2\pi} |f(r e^{it})|^p
 \frac{dt}{2\pi}  <\infty \,.
\]
We refer the reader to the book \cite{D} for further information about Hardy spaces (and the Nevanlinna class).
\par\smallskip
For $f \in H^\infty (\D)$ we let $\widetilde{f}\colon \partial \D \to \C$  be the corresponding boundary map as in \eqref{eq-radiallimit}.
Recall that the radial limit is defined a.e. $\theta \in [0, 2\pi)$ and that $\|f\|_{H^\infty (\D)}=\|\widetilde{f}\|_{L^\infty (\partial \D )}$ .
Also note that the operation of taking the radial limit is a multiplicative homomorphism. We will look at
function spaces on the boundary $\partial \D$. Therefore we will induce different kinds of objects acting on $\partial \D$ from the corresponding objects
acting on the disk $\D$ via the radial map and its $1$-to-$1$ correspondence in the a.e. sense. In what follows, the induced objects are distinguished by the symbol $\widetilde{\ }$ from the original ones. In this sense we will consider $H^\infty (\D) \subset L^\infty (\partial \D)$ isometrically as a
Banach subalgebra.

\subsection{Uniform algebras}

Recall that a unital commutative Banach algebra $\A$ is a \emph{uniform algebra} if 
\[ \|a^2 \| = \|a\|^2,\quad a\in \A.\]

\subsection{Douglas algebras}

\par\smallskip
Let $\A$ be a uniformly closed algebra with $H^\infty\subset \A\subset L^\infty$. For instance, consider any set $Q$ of inner functions in $H^\infty$ and take $\A=[H^\infty, \overline Q]$, the closed algebra generated by $H^\infty$ and $\overline Q$. Note that such algebra $\A$ is simply the norm closure of
\[
\left\{f \overline{b_1}^{n_1}\cdots\overline{b_k}^{n_k}\colon f\in H^\infty, b_1, \ldots, b_k\in Q  \right\}\subset L^\infty \,.
\]
\par
Any algebra of the form $[H^\infty, \overline Q]$, where $Q\subset\{u\colon u\text{ is an inner function in } H^\infty\}$ is called a \emph{Douglas algebra}.
\par
The simplest example $[H^\infty, \overline z]$ coincides with the closed linear space $H^\infty + \cc \subset L^\infty$, where $\cc$ denotes the continuous functions on the unit circle. It is known that any Douglas algebra that properly contains $H^\infty$ must contain $H^\infty + \cc$ as well (see \cite{Sa}). Moreover, as it is proved in \cite{CC}, $H^\infty + \cc=[H^\infty, \overline B]$ if and only if $B$ is a finite Blaschke product.  Note that if $Q=\{u\colon u\text{ is an inner function in } H^\infty\}$, then $[H^\infty, \overline Q]=L^\infty$.
\par\smallskip
The main theorem in \cite{Mar-acta} combined with a result by S.~Y.~Chang \cite{Chang} proves a conjecture of R. Douglas: every uniformly closed algebra $\A$,
$H^\infty \subset \A \subset L^\infty$, is generated by $H^\infty$ itself and by the set $\{u\in \A\colon u\text{ is an inner function in } H^\infty\}$. In other words, every uniformly closed subalgebra between $H^\infty$ and $L^\infty$ is a Douglas algebra.
\par

\par

\subsection{Further algebras built from $H^\infty$}
Next we will use the above ideas to gain access to more Nevanlinna class style spaces.
Suppose $\X \subset H^\infty=H^\infty(\D)$ is a subalgebra that contains some inner function. Denote by $N_\X$ the set
\[
N_{\X} := \overline{\left\{\frac{\widetilde{f}}{\widetilde{g}} \in L^\infty (\partial \D )\colon f,g\in \X,\ g\ \mathrm{inner}\right\}}\subset L^\infty=L^\infty(\partial \D).
\]
Note that, as was mentioned in the previous section, $N_{H^\infty}=L^\infty$. Also, that if $\frac{\widetilde{f_1}}{\widetilde{g_1}},\frac{\widetilde{f_2}}{\widetilde{g_2}}\in N_{\X}$ and $c\in \C$, then the operations are defined
point-wise a.e. on $\partial \D$ as follows:
\[c\frac{\widetilde{f_1}}{\widetilde{g_1}} := \frac{c\widetilde{f_1}}{\widetilde{g_1}},\quad \frac{\widetilde{f_1}}{\widetilde{g_1}}+\frac{\widetilde{f_2}}{\widetilde{g_2}}:=\frac{\widetilde{f_1} \widetilde{g_2} + \widetilde{f_2} \widetilde{g_1}}{\widetilde{g_1} \widetilde{g_2}},\quad
\frac{\widetilde{f_1}}{\widetilde{g_1}}\cdot \frac{\widetilde{f_2}}{\widetilde{g_2}}:=\frac{\widetilde{f_1} \widetilde{f_2}}{\widetilde{g_1} \widetilde{g_2}}
\in N_{\X}.\]
Also,
\[\left\|\frac{\widetilde{f_1}}{\widetilde{g_1}} + \frac{\widetilde{f_2}}{\widetilde{g_2}}\right\|_{L^\infty}\leq \left\|\frac{\widetilde{f_1}}{\widetilde{g_1}}\right\|_{L^\infty } + \left\|\frac{\widetilde{f_2}}{\widetilde{g_2}}\right\|_{L^\infty},\quad
\left\|\frac{\widetilde{f_1}}{\widetilde{g_1}} \cdot \frac{\widetilde{f_2}}{\widetilde{g_2}}\right\|_{L^\infty}\leq \left\|\frac{\widetilde{f_1}}{\widetilde{g_1}}\right\|_{L^\infty} \left\|\frac{\widetilde{f_2}}{\widetilde{g_2}}\right\|_{L^\infty }.\]
Consequently, the operations $\cdot$ and $+$ are coordinate-wise bounded bilinear operations inherited from $L^\infty$. Thus
$N_{\X}$ is a Banach algebra.
\par

\section{Some auxiliary results}

Let $\phi_a$ be the involutive automorphism of the unit disk $\D$ which interchanges the points $a\in \D$ and $0$ defined by
\begin{equation}\label{eq-autom}
\phi_a(z)=\frac{a-z}{1-\overline a z}\,,\quad z\in\D\,.
\end{equation}
It is an easy consequence of the Schwarz lemma that all automorphisms of $\D$, this is, M\"{o}bius transformations from the unit disk onto itself, have the form $\lambda \phi_a$ for some $|\lambda|=1$. Note also that any such mapping $\phi_a$ induces a bilipschitz homeomorphism $\partial \D \to \partial \D$ and that composition $f\mapsto f \circ \phi_a$ is a rotation on $N_\X$ if $\X$ is invariant under compositions from the inside with such automorphisms.

Another relevant rotation on $N_\X$ is defined by $f\mapsto \psi_\z f$, where $\z\in \partial \D$ and $(\psi_\z f)(z)=f(z\overline\z)$, $z\in \partial \D$. Of course, this extends naturally to a rotation on the whole of $L^\infty (\partial \D )$.
\subsection{A property of automorphisms of the unit disk}

We use, as usual, the following notation for the \emph{average integral operator }
\[
\psi\to \dashint_{\partial \D } \psi(\theta)  d\theta= \frac{1}{2\pi} \int_0^{2\pi} \psi(e^{i\theta}) \ d\theta\,,
\]
where $\psi$ is a measurable function on the unit circle.
\begin{lemma}\label{lm: unit_spread}
Let $\phi_a$ be the automorphism of the unit disk given by \eqref{eq-autom}. Suppose that $f\in H^\infty$ with $\|f\|_{\infty} =1$. Then
\begin{equation}\label{eq-lm:unit spread}
\sup_{a \in \D}\, \left|\,\dashint_{\partial \D } (f\circ \phi_{a})(\theta)\ d\theta \right|=1.
\end{equation}
\end{lemma}
\begin{proof}
We are to check that 
\begin{equation}\label{eq-lm:unit spread2}
\sup_{a \in \D} \left|\frac{1}{2\pi}\int_0^{2\pi} (f\circ\phi_{a})(e^{i\theta}) \ d\theta \right|=1.
\end{equation}
For a real positive number $0<r<1$, we denote by $(f\circ\phi_{a})_r$ the \emph{dilation function}
\[
(f\circ\phi_{a})_r(z)=(f\circ\phi_{a})(rz)\,,\quad z\in\D\,.
\]
Using the fact that both $f$ and $\phi_a$ are analytic in the unit disk, we get that $(f\circ\phi_{a})_r$ are analytic functions (hence harmonic) in the closed unit disk for each all such values of $r$. Therefore, using the mean value property we get
\begin{equation*}\label{eq-integralautom}
\frac{1}{2\pi}\int_0^{2\pi} (f\circ\phi_{a})_r(e^{i\theta}) \ d\theta = (f\circ\phi_{a})_r(0)=f(a)\,,
\end{equation*}
which gives, by the application of the dominated convergence theorem,
\begin{equation*}\label{eq-integralautom}
\frac{1}{2\pi}\int_0^{2\pi} (f\circ\phi_{a})(e^{i\theta}) \ d\theta= f(a)\,.
\end{equation*}
This readily implies \eqref{eq-lm:unit spread2} (and then \eqref{eq-lm:unit spread} as well) since
\[
1=\|f\|_{\infty} =\sup_{a\in\D} |f(a)|\,.
\]
\end{proof}

\subsection{Abstract harmonic analysis tool}

\begin{proposition}\label{lm: unit_average}
Suppose that $f \in L^\infty(\T)$ with $\|f\|_{L^\infty} =1$. Assume also that $\dashint_{\T} f = s$ where $0\leq s \leq 1$. Then
\[s \1_{\T} \in \overline{\conv}^{\omega^*} (\psi_\zeta f\colon \zeta \in \T )\subset L^\infty (\T)\]
where $\psi_\zeta (f) (z) = f(z\overline\z)$, $\zeta, z\in \T$.
\end {proposition}

The above statement is rather easy to see for exponents $1\leq p<\infty$ but the case $ p=\infty$, treated below, is more complicated. We do not know if the statement holds for norm topology in place of the weak-star topology.
We note that if it does hold, then our main result (Theorem~\ref{thm: main} below) in the case of $\A=L^\infty$ can be improved to the full convex-transitive case.
Also, if this improvement is possible, the next natural question is whether a similar conclusion holds for a general compact group (or an amenable locally compact group in the weak-star setting) with the Haar measure.

\begin{proof}
Fix $f$ and $s$ as above and assume to the contrary that
\[s \1_{\T} \notin \overline{\conv}^{\omega^*} (\psi_\z f\colon \z\in \T )\subset L^\infty (\T). \]
Then there exists by the Hahn-Banach theorem a functional $F \in L^\infty (\T )^*$, $\|F\|=1$, such that
\[\sup F( \overline{\conv}^{\omega^*} (\psi_\z f\colon \z\in \T ))< F(s \1_{\T} ).\]
Moreover, we may pick $F$ to be weak-star continuous above by using the Hahn-Banach separation on locally convex spaces (see e.g. \cite[Thm. 4.25]{fhhmz}), in this case for the weak-star topology.

Denote by $G_n$, where $n$ is a positive integer, the finite cyclic subgroup of $\T$ generated by $e^{i\frac{2\pi}{2^n}}$ (in $\T$).
Define linear operators
$T_n \colon L^\infty (\T) \to L^\infty (\T )$ by
\[(T_n f)(z) = \frac{1}{\# G_n}\sum_{g\in G_n} f(gz).\]
Clearly $T_n (\1_{\T})=\1_{\T}$ and $\|T\|=1$. The adjoint operators $T_{n}^* \colon L^\infty (\T )^* \to L^\infty (\T )^*$ are given by
\[(T_{n}^* x^*)[x] = x^* (T_n x),\quad x\in L^\infty (\T ),\ x^* \in L^\infty (\T )^* .\]

By the Banach-Alaoglu theorem, $\B_{L^\infty (\T )^*}$ is weak-star compact.
Thus the sequence $T_{n}^* F$ has a weak-star cluster point, say $F_0$.
This, of course, need not be unique, a priori. Note that $F_0 (\1_{\T})=F(\1_{\T})$ by the construction.

As an element of the dual of $L^\infty (\T )$ the element $F_0$ can be seen as a finitely additive signed measure $\Sigma\to \C$ with finite variation, $\|F_0 \|$, see e.g. \cite[IV.8.16]{DS}.

However, note that $F_0$ has the property that for all $n\geq 1$,
\[F_0 x = F_0 T_n x\,.\]
Thus, by using the finite additivity of the measure we get that for any dyadic decomposition of the form
\[
\bigcup_{k=0}^{2^n -1} D_{k,n} = \T\,,\quad \text{where}\quad
D_{k,n}=\left[e^{ik\frac{2\pi}{2^n}},e^{i(k+1)\frac{2\pi}{2^n}}\right)\,,\quad k=0,\ldots 2^n-1\,,\]
we have $F_0 (\1_{D_{i,n}} \cdot) = F_0 (\1_{D_{j,n}} \cdot)$ for any $0\leq i,\ j\leq 2^n-1$.
In particular, $\|F_0 (\1_{D_{i,n}} \cdot) \|$ depends only on $n$.

By applying an outer measure approximation of measurable sets we see that
\[\frac{m(M)}{2\pi} \|F_0\| =\|F_0 (\1_{M} \cdot)\|\]
for general measurable sets $M \subset \T$. Indeed, for each such set $M$ and any $\varepsilon>0$, there is
a positive integer $n$ and there are finitely many $n$-th level dyadic intervals $D_{i,n}$ as above such that
\[m(M\ \Delta\ \bigcup_i D_{i,n} )< \varepsilon\,,\]
where $\Delta$ denotes the symmetric difference between sets, that is, $A\Delta B= (A\setminus B) \cup (B\setminus A)$.

On the other hand, it follows from the weak-star continuity of $F$ that
\[\limsup_{m(E)\to 0}\|F (\1_{E} \cdot)\|=0\]
($F\in L^1 (\T)$) and since the above limit is uniform and
\[\|T_{n}^* F (\1_{E} \cdot)\| \leq \sup_{m(E')=m(E)} \| F (\1_{E'} \cdot)\|\,,\]
we obtain
\[\limsup_{m(E)\to 0}\| F_0 (\1_E \cdot)\|=0.\]

Clearly $F_0$ is an absolutely continuous $\sigma$-additive signed measure $\Sigma \to \C$.
Thus we may regard $F_0 \in L^{1}(\T)$ via the Radon-Nikodym derivative.

By a standard argument employing Lusin's theorem, we can see that $\|f-\psi_\z (f) \|_1 \to 0$
as $\T\ni \z\to 1$. Since $T_{n}^* F_0 =F_0$ for $n\geq 1$ we obtain that $F_0 = \psi_\z F_0$ for every
$\z\in \T$ (up to a modification in a null set). That is, $F_0 = c \1_{\T}\in L^{1}(\T)$
can be regarded as a constant function.

The constant is given by $c=F(\1_{\T})$, so that $F_0$ is in fact unique and $T_{n}^* F \stackrel{\omega^*}{\longrightarrow} F_0$.

Observe that
\[\sup_{|z|=1} F(\psi_z f)\geq F_0 (f)=c\dashint_{\T} f =cs\]
and
\[F(\1_{\T})=(T_{n}^* F)(\1_{\T})\to F_0 (\1_{\T})=cs.\]
Thus
\[\sup_{|\z|=1} F(\psi_\z f)\geq F(\1_{\T})=cs,\]
contradicting the counter assumption. This concludes the argument.
\end{proof}

\section{Main results}

\begin{theorem}\label{thm: main}
Let $\A$ be a Douglas algebra, $H^\infty \subset \A \subset L^\infty$, which is invariant under compositions from the inside with automorphisms of the unit disk.
Then the following conditions are equivalent:
\begin{enumerate}
\item The Douglas algebra is not minimal, i.e. $\A\neq H^\infty$,
\item $\A$ is $\omega^*$-convex-transitive,
\item $\A$ is $\omega^*$-convex-transitive with respect to the subgroup of weighted composition operators of the type $\widetilde{f}\mapsto \widetilde{\psi} \widetilde{(f\circ \phi)}$ where $\widetilde{\psi}\in A$ is a.e. unimodular and $\phi$ is an automorphism of $\D$.
\end{enumerate}
\end{theorem}

Remarks: It is known that $L^\infty$  is convex-transitive as a Banach space.
The point here is that the $\omega^*$-convex-transitivity holds even if we restrict to a much smaller isometry group, which, in addition to being multiplicative, splits to very particular type of isometries motivated by function theory. Recall that we denote by $\omega^*$ the locally convex topology $\sigma(\A,L^1 )$
which is the weak-star topology in the case that $\A=L^\infty$.

\begin{proof}
The rotations of $H^\infty$ are of the type \eqref{eq: char}, which clearly map the constant functions to constant functions.
It is easy to see that the subspace of constant functions in $L^\infty$ is a $1$-dimensional $\omega^*$-closed subspace.
Thus $H^\infty$ is not $\omega^*$-convex-transitive.

The rest of the argument we will study Douglas algebras $\A\supsetneq H^\infty$. As it was mentioned before, in such a case $\A$ contains $\cc$, the continuous functions on $\partial \D$, as a subspace. To prove the statement it suffices to show the $\omega^*$-convex-transitivity of $\A$ with respect to the \emph{special} isometry subgroup, which is denoted by $\mathcal{G}_\A$, in what follows.

We will achieve this in two main steps. First, we will show that
for any $x\in \S_{\A}$ the unit function $\1_{\partial \D}$ is contained in
$\overline{\conv}^{\omega^*}(\mathcal{G}_{\A}(x))$. Second, we check that for any $y \in \S_{\A}$ we
have that $y\in \overline{\conv}^{\omega^*}(\mathcal{G}_{\A}(\1_{\partial \D}))$. Then a standard argument  (see Lemma \ref{lm: tri}) yields that $y\in \overline{\conv}^{\omega^*}(\mathcal{G}_{\A}(x))$ for any $x\in \S_{\A}$ and the
$\omega^*$-convex-transitivity of the space follows.
\par\smallskip
First step. Let $x\in \S_{\A}$. We will show that $\1_{\partial \D} \in \overline{\conv}^{\omega^*}(\mathcal{G}_{\A}(x))$.
Let $\varepsilon >0$. Then by the construction of the Douglas algebra there is $x_0 = \frac{\widetilde{f}}{\widetilde{g}}$ with
$f,g \in H^\infty$ and $g$ inner, such that $\|x-x_0 \|<\varepsilon$. Thus, it actually suffices to show that
$\1_{\partial \D} \in \overline{\conv}^{\omega^*}(\mathcal{G}_{\A}(x_0 ))$.

It is easy to see that the multiplication of functions in $\A$ by inner functions $\widetilde{g} $ such that $\overline{\widetilde{g}} \in \A$
induces a rotation on $\A$. Therefore we may write $y=T(x_0 )$, $T\in \mathcal{G}_{\A}$, such that $y$ is the boundary value of an analytic function $f\in \S_{H^\infty (\D)}$.

According to Lemma \ref{lm: unit_spread} we have that
\[\sup_{a \in \D}\, \left|\,\dashint_{\partial \D } f\circ \phi_{a}(\theta)\ d\theta \right|=1\]
and therefore we obtain sequences $a_n \in \D$ and $c_n \in \partial \D$ such that
\[\lim_{n\to\infty} c_n \dashint_{\partial \D}    f\circ \phi_{a_n}(\theta)\  d\theta =1.\]

Observe that composition of functions in $\A$ from the inside by M\"obius transformations $\phi$ of the unit disk onto itself induces
rotations on $\A$. Indeed, if $f,g\in H^\infty (\D)$, then the essential supremum of the modulus of boundary values
$\frac{\widetilde{f \circ \phi}}{\widetilde{g \circ \phi}}$ coincides with that of $\frac{\widetilde{f}}{\widetilde{g}}$, since
$\phi$ induces a bilipschitz transform of the boundary onto itself.
We conclude that
\[\sup S \{Tx\colon T \in \mathcal{G}_{\A}\}=1\,,\]
where $S\colon \A\to \R$ is the average integral operator.

Thus, Proposition \ref{lm: unit_average} yields
\[\1_{\partial \D} \in \overline{\conv}^{\omega^*} (Rx\colon R\in \mathcal{G}_{\A})\,,\]
where we use rotations of the disk composed both from the inside and outside.
\par\smallskip
Second step. We first treat the case $H^\infty \subsetneq \A \subseteq L^\infty$, case a), and then give a superfluous argument for the case with
$\A=L^\infty$, case b).
\medskip

{\it Case a).} Note that if $v \in \cc$ is unimodular, then $\overline{v} \in \cc$ is unimodular as well. Moreover, multiplication
from the outside by $v$ induces a rotation on $\A$. It is known that the norm closed convex hull of unimodular functions of $\cc$ is $\B_\cc$.
This can be verified for example by using the Russo-Dye theorem.

Next we show that $\overline{\B_{\cc}}^{\omega^*} = \B_{L^\infty (\partial \D)}$.
Indeed, by mollifying any $z \in L^\infty (\partial \D)$ we can approximate it in the $L^1$-sense by a sequence $(z_n)$ of continuous functions
with $\|z_n \|_\infty \leq \|z\|_\infty$. In particular $z_n \to z$ in measure and we see easily that
\[\int_{\partial \D} h (z_n - z) \to 0\]
as $n\to\infty$ for any $h\in L^1 (\partial \D)$. Thus $z_n \stackrel{\omega^*}{\longrightarrow} z$ and hence
$\overline{\B_{\cc}}^{\omega^*} = \B_{L^\infty }$.
It follows that $\B_\A =\overline{\conv}^{\omega^*} (\mathcal{G}_\A (\1))$.
\medskip

{\it The case b)} ($\A= L^\infty$). Fix $\frac{\widetilde{f}}{\widetilde{g}}\in \S_{\A}$, where $g$ is inner. Let us denote by $I_1$ and $I_2$ the radial parts of some given inner functions. Note that multiplication with functions of the form $y \mapsto \frac{I_1}{I_2}y$ defines a rotation on $\A$.
The fact that this operator is linear is clear, it is isometric by virtue of the properties of inner functions
and invertibility follows by noting that $\frac{I_2}{I_1}$ defines also a linear isometry.
First note that
\begin{equation}\label{eq: fg}
\frac{\widetilde{f}}{\widetilde{g}}\in \mathcal{G}_{\A}(\widetilde{f}).
\end{equation}
By  applying Theorem \ref{thm: dense} we obtain that
\[\widetilde{f} \in \overline{\conv}(\widetilde{B}\colon\ B\ \mathrm{Blaschke\ product}),\]
which reads
\[\widetilde{f} \in \overline{\conv}(\mathcal{G}_{\A}(\1_{\partial \D})),\]
since in the case of the maximal Douglas algebra the multiplications from the outside by Blaschke products induce rotations.
Thus, by \eqref{lm: tri} we conclude that
\[\frac{\widetilde{f}}{\widetilde{g}} \in \overline{\conv}(\mathcal{G}_{\A}(\1_{\partial \D})).\]

By combining the two steps the principle in Lemma~\ref{lm: tri} yields that $\A$ is $\omega^*$-convex-transitive with respect to the special subgroup of rotations.
\end{proof}
Note that both the smallest Douglas algebra that properly contains $H^\infty$, this is $H^\infty+\cc$, and the maximal one, $L^\infty$, are invariant under compositions from the inside with surjective M\"obius transformations of the disk. Not every Douglas algebra satisfies this condition (see \cite[Cor. 6]{GGG}). However, in the next examples we show two Douglas algebras $\A_1$ and $\A_2$ satisfying the hypotheses in Theorem~\ref{thm: main} and with $H^\infty+\cc \varsubsetneq \A_i \varsubsetneq L^\infty$, $i=1,2$.

\begin{example}
An analytic function $f\in H^\infty$ with $\|f\|_\infty=1$ is said to have angular derivative at $\z\in\partial \D$ if there exists a point $\omega\in\partial\D$ such that $(f(z)-\omega)/(z-\z)$ has finite non-tangential limit as $z\to\z$ (see \cite[p. 50]{CoMc}). Note that if $f$ is analytic at $z=\zeta$ and $|f(\zeta)|=1$, then $f$ has angular derivative at $\zeta$. By using the Julia-Carath\'{e}odory theorem \cite[p. 51]{CoMc}, it is easy to check that $f$ has angular derivative at every point on $\partial\D$ if and only if for every automorphism $\phi$ of the unit disk the function $f\circ\phi$ also has angular derivative on $\partial\D$.
\par
Consider
\[
Q_1=\{b\colon b\text{ is a Blaschke product with angular derivative at every point on } \partial\D\}
\]
and let $\A_1=[H^\infty, \overline{Q_1}]$. Since both $H^\infty$ and $Q_1$ are invariant under composition with surjective M\"obius transformations from the inside, so is $\A_1$.
\par
Theorem 11 and Example 1 in \cite{GGG} directly prove that $\A_1 \varsubsetneq L^\infty$. Now, define the sequence $\{a_n\}$ of points in the unit disk by
\[
a_n=\frac{1}{n^2+1}+\frac{n^2e^{\frac in}}{n^2+1}\,,\quad n=1\,, 2\,,\ldots
\]
and consider the corresponding (infinite) Blaschke product $b_1$ as in \eqref{eq-Blaschkeprod} with zeros $\{a_n\}$.
\par
Since $\zeta=1$ is the only set of accumulation points of $\{a_n\}$, by \cite[Ch. II, Thm. 6.1]{G} we have that $b_1$ extends to be analytic on the complement of
\[
\{1\}\cap \{1/\overline{a_n}\colon n=1, 2, \ldots\}\,.
\]
Hence, $b_1$ has angular derivative at $\z$ for all $\z\neq 1$. To prove that $b_1$ has angular derivative at $\zeta=1$ as well, we use Frostman's theorem \cite{Frostman2} which says that a Blaschke product with zeros $\{a_n\}$ has angular derivative at $\z=1$ if and only if
\begin{equation}\label{eq-frostman}
\sum_{n=0}^\infty \frac{1-|a_n|^2}{|1-a_n|^2} <\infty\,.
\end{equation}
A straightforward calculation shows that $\frac{1-|a_n|^2}{|1-a_n|^2}=1/n^2$. This proves that $b_1\in Q_1$, hence $[H^\infty, \overline{b_1}]\subset \A_1$. Since $H^\infty+\cc=[H^\infty, \overline{b}]$ if and only if $b$ is a finite Blaschke product (see \cite{CC}) we conclude that $H^\infty+\cc \varsubsetneq \A_1$.
\end{example}

\begin{example}
For $z, w\in\D$, we let
\[
\rho(z,w)=\left|\frac{z-w}{1-\overline w z}\right|
\]
be the \emph{pseudo-hyperbolic distance} of $z$ and $w$. It is easy to check that for any automorphism $\phi$ of the disk the following equality holds for any pair of points $z, w$ in $\D$:
\[
\rho(\phi(z), \phi(w))=\rho(z,w)\,.
\]
\par
A infinite Blaschke product of the form \eqref{eq-Blaschkeprod} is called an \emph{interpolating Blaschke product} if its zero set $\{a_n\}_{n=1}^\infty$ satisfies
\[
\inf_n \prod_{n\neq m} \rho(a_n, a_m)=\delta >0\,.
\]
In the case when
\[
\lim_{n\to\infty} \prod_{n\neq m} \rho(a_n, a_m)=1\,,
\]
we say that the Blaschke product is \emph{thin} (or sparse). Note that $B$ is a thin Blaschke product if and only if $B\circ \phi$ is a thin Blaschke product for any surjective M\"obius transformation $\phi$ of the disk.
\par

Define $Q_2=\{b\colon b\text{ is a thin Blaschke product}\}$ and consider $\A_2=[H^\infty, \overline{Q_2}]$. Again, both $H^\infty$ and $Q_2$ are invariant under compositions with surjective M\"obius transformations from the inside, hence, so is $\A_2$. This algebra $\A_2$ was analyzed in \cite{H} where the author proves the equality $\A_2=H^\infty +\bb$, $\bb$ being  the smallest (closed) $C^*$ subalgebra of $L^\infty$ containing $Q_2$; that is $\bb=[Q_2,\overline{Q_2}]$. Note that this directly proves that $H^\infty+\cc \varsubsetneq \A_2$, however, we can give another proof of this latter fact as follows.
\par
Consider the Blaschke product $b_2$ with zeros at the points $a_n=1-1/2^{2^n}, n=1, 2, \ldots$ Since
\[
\lim_{n\to\infty} \frac{1-|a_{n+1}|}{1-|a_n|}=0\,,
\]
we get that $b_2$ is a thin Blaschke product (see the comment after Proposition~1.1 in \cite{GM}) and therefore, $[H^\infty, \overline{b_2}]\subset \A_2$.
\par
On the other hand, the series in \eqref{eq-frostman} for our choice of $\{a_n\}$ diverges. Thus, $b_2$ is an (infinite thin) Blaschke product which does not have angular derivative at $\z=1$.
\par
Bearing in ming that $b_2$ is an infinite Blaschke product, we can argue as in the previous example to get that $H^\infty+\cc \varsubsetneq \A_2$.
\par\smallskip
It was proved in \cite{H} that any Blaschke product $b$ with $\overline{b}\in \A_2$ must be a finite product of thin Blaschke products. Not every Blaschke product has this property so that $\A_2 \varsubsetneq L^\infty$.

\end{example}

\begin{theorem}\label{thm: sub1}
Let $Y \subset H^\infty (\D)$ be a subspace which is preserved as a set when composing the functions with surjective M\"obius transformation of the disk from the inside.
Let $\A\subset L^\infty (\partial \D)$ be the sub-$C^*$-algebra generated by $Y$. Then $\A$ is $\omega^*$-convex-transitive with respect to the group of weighted composition operators of the form
\[\widetilde{f} \mapsto \widetilde{u} \widetilde{(f\circ \phi)}\]
where $\widetilde{u}\in \A$ is a.e. unimodular and $\phi$ is a M\"obius transformation of the unit disk onto itself.
\end{theorem}

\begin{proof}
The argument is an adaptation of the proof  of Theorem \ref{thm: main}.

It is clear that the a.e. unimodular functions of $\A$ are exactly all the unitary elements. By the $C^*$-structure of the space this set is invariant
under complex conjugation. Therefore multiplication from the right with a unimodular function induces a surjective linear isometry on $\A$.
Then the Russo-Dye theorem implies that the closed convex hull of these points is the closed unit ball of $\A$.
We will use this fact analogously as in the application of Theorem \ref{thm: dense}.

The rest of the argument runs similarly as in the proof of Theorem \ref{thm: main}.
\end{proof}

Observe that if $Y$ contains the identity mapping $z\mapsto z$, then the Stone-Weierstrass theorem yields that $\A$ then already contains the space of continuous functions
$\cc$. This means that there are typically quite many unimodular functions.

\begin{remark}\label{thm: sub2}
Let $\Y \subset H^\infty (\D)$ be a closed unital subalgebra generated by a set of inner functions and conformally invariant. 
Then $N_{\Y}$ is a uniform algebra $\omega^*$-convex-transitive with respect to the group of weighted composition operators of the form
\[\widetilde{f} \mapsto \widetilde{\psi} \widetilde{(f\circ \phi)}\]
where $|\widetilde{\psi}|=1$ a.e. and $\phi$ is an automorphism of the unit disk.
\par
The argument is similar as above, applying the fact that the closed convex hull of
$\{a\overline{b}\colon a,b \in \Y \ \mathrm{inner}\}$ is the closed unit ball of $N_{\Y}$, see \cite{Mar} or \cite[p. 195]{G}.
\end{remark}

For example, if $\Y$ above coincides with the subalgebra generated by finite Blaschke products, then $N_{\Y}=\cc$.
The convex-transitivity of this complex space is known, albeit not with respect to such an
isometry subgroup.

\begin{example}
Let $\Y \subset H^\infty$ be the subset consisting of elements $y \in H^\infty$ such that $y$ extends continuously to $m$-almost every $\theta \in \partial\D$
(using the radial limits). This set clearly includes the disk algebra and it also contains many Blaschke products (interpolating or not) and singular inner functions. It is easy to verify that
$\Y$ is a Banach algebra. Let $\rm Z$ be any conformally invariant Banach algebra generated by $\{I, \overline{I}\colon I \in \mathcal{I}\}$ where $\mathcal{I}$ is
a collection  of inner functions of $\Y$ containing $\{1,z\}$. Then $\rm Z$ is convex-transitive (in the norm topology) with respect to the isometry group of weighted composition operators, similarly as above. By virtue of the regularity property of the space it follows easily that $\1_{\partial \D} \in \overline{\conv} (\mathcal{G}_{\rm Z} (x))$ for any $x \in \S_{\rm Z}$. On the other hand, by the previous remark we get $y \in  \overline{\conv} (\mathcal{G}_{\rm Z} (\1_{\partial \D} ))$ for any $y \in \S_{\rm Z}$.

Note that $\rm Z$ consists of functions extending continuously to the boundary almost everywhere. Therefore $H^\infty \not\subset \rm Z$.
We suspect that if $\mathcal{I}$ is the collection of all inner functions of $Y$, then the corresponding $\rm Z$ coincides with $N_Y$ (so that $N_Y$ would be convex-transitive).
\end{example}

\section{Discussion}

After Proposition \ref{lm: unit_average} we raised the question whether it suffices to take the
norm closure of the convex hull in the statement. We feel this question is rather intriguing. Recall that in the proof of this proposition we were able to pick a weak-star continuous functional $F$ in the Hahn-Banach separation by virtue of the weak-star closure. We now show how the argument applied fails if weak-star continuity is not imposed.

\begin{example}
There exists $F \in (L^\infty (\T))^*$, $\|F\|=1$, $F \notin L^1 (\T)$, (in particular, $F\neq \1_{\T}$) such that
\[F(f)=F(\psi_\zeta f)\quad \mathrm{for\ all}\ \zeta=e^{i\frac{2\pi}{n}},\ f\in L^\infty (\T) ,\ n\in\N .\]
\end{example}

Indeed, let $\mathcal{I}$ be the collection of all measurable subsets $I\subset \T$ having strictly positive Lebesgue measure.
Let $G_n \subset \T$, $n \geq 1$, be finite subgroups generated by $\{e^{i\frac{2\pi}{k}}\colon 1\leq k\leq n\}$.

For each $k\in\N$ let $I_k \in \mathcal{I}$ be a $G_k$-invariant set with measure $0<m(I_k ) \leq 2^{-k}$. For each $n\in\N$ let
\[J_n = \bigcup_{k=n}^\infty I_k .\]
Note that the $J_n$ sets are $G_n$-invariant and have measure $0< m(J_n ) \leq 2^{-n+1}$.

Let $\mathcal{U}$ be a free ultrafilter over $\N$ and put
\[F(f):=\lim_{n,\mathcal{U}} \dashint_{J_n} f,\quad f\in L^\infty (\T), \]
(limit with respect to ultrafilter, see e.g. \cite{Heinrich}). It is easy to see that this defines an element $F\in (L^\infty (\T))^*$ with the required properties.

\subsection{Acknowledgments}
This research was supported by Academy of Finland Project \#268009. The first named author was also supported by Spanish MINECO Research Project MTM2012-37436-C02-02 and the second one by Finnish V\"{a}is\"{a}l\"{a} Foundation's Research Grant.

\end{document}